\documentclass[reqno,12pt]{amsproc}

\usepackage[top=3cm, bottom=3cm, left=3cm, right=3cm]{geometry}
\usepackage{amsmath,amsthm,amscd,amssymb,bbm,mathrsfs,latexsym,tikz,enumerate,verbatim,multirow}
\usepackage[colorlinks,citecolor=blue,pagebackref,hypertexnames=false]{hyperref}

\allowdisplaybreaks

\newcommand{\mR}{\mathbb{R}}

\def\Exp(#1){{\mathbb E}(#1)}

\def\inpr#1,#2{\t \hbox{\langle #1 , #2 \rangle} \t}
\def\ip<#1,#2>{\langle #1,#2 \rangle}

\def\norm#1{\left \Vert #1 \right \Vert}
\def\paren(#1){\left( #1 \right)}

\def\sparen(#1){\Bigl ( #1 \Bigr )}
\def\ssparen(#1){ (#1) }
\def\st{\thinspace : \thinspace}

\def\v3(#1,#2,#3){\begin{pmatrix} #1 \\ #2 \\ #3 \end{pmatrix}}

\newcommand{\vecx}{{\boldsymbol{x}}}
\newcommand{\vecy}{{\boldsymbol{y}}}

\newcommand{\bxi}{\boldsymbol{\xi}}

\newcommand{\zero}{\boldsymbol{0}}

\numberwithin{equation}{section}
\newtheorem{theorem}{\bf Theorem}[section]

\newtheorem{corollary}[theorem]{\bf Corollary}

\theoremstyle{definition}
\newtheorem{definition}[theorem]{\bf Definition}

\theoremstyle{remark}
\newtheorem{remark}[theorem]{\bf Remark}
\numberwithin{equation}{section}

\begin{document}
\title{Polygonal equalities and $p$-negative type}

\author{Ian Doust} 
\address{School of Mathematics and Statistics, University of New South Wales, Sydney, New South Wales 2052, Australia}
\email{i.doust@unsw.edu.au}

\author{Anthony Weston}
\address{Arts and Sciences, Carnegie Mellon University in Qatar, Education City,
PO Box 24866, Doha, Qatar}
\email{aweston2@andrew.cmu.edu}



\begin{abstract}
Nontrivial $p$-polygonal equalities impose certain conditions on the geometry of a metric space $(X,d)$ and so it is of interest to be able to identify the values of $p \in [0,\infty)$ for which such equalities exist. Following work of Li and Weston
\cite{LiW}, Kelleher, Miller, Osborn and Weston
\cite{KMOW} established that if a metric space $(X,d)$ is of $p$-negative type, then $(X,d)$ admits no nontrivial $p$-polygonal equalities if and only if it is of strict
$p$-negative type. In this note we remove the
underlying premise of $p$-negative type from this theorem.
As an application we show that the set of all $p$ for which a
finite metric space $(X,d)$ admits a nontrivial $p$-polygonal
equality is always a closed interval of the form $[\wp, \infty)$,
where $\wp > 0$, or the empty set. It follows that for each
$q \not= 2$, the Schatten $q$-class $\mathcal{C}_{q}$ admits a
nontrivial $p$-polygonal equality for each $p > 0$. Other
spaces with this same property include $C[0, 1]$ and
$\ell_{q}^{(3)}$ for all $q > 2$.
\end{abstract}

\maketitle


\section{Introduction}\label{section-1}

Motivated by Enflo's definition of generalized roundness \cite{En1} and results of Li and Weston \cite{LiW},
Kelleher, Miller, Osborn and Weston \cite{KMOW} introduced the concept of a $p$-polygonal equality in a metric space $(X,d)$.

\begin{definition}
Suppose that $(X,d)$ is a metric space and that $p \ge 0$. A \textbf{$p$-polygonal equality} in $(X,d)$ is an equality of the form
  \begin{equation}\label{p-poly} 
  \sum_{i=1}^s \sum_{j=1}^t m_i n_j d(x_i,y_j)^p
    = \sum_{1 \le i_1<i_2 \le s} m_{i_1} m_{i_2} d(x_{i_1},x_{i_2})^p
     + \sum_{1 \le j_1<j_2 \le t} n_{j_1} n_{j_2} d(y_{j_1},y_{j_2})^p
  \end{equation}
where $x_1,\dots,x_s,y_1,\dots,y_t$ are a collection of
(not necessarily distinct) points in $X$, together with real
number weights $m_1,\dots,m_s,n_1,\dots,n_t$ that satisfy
$m_{1} + \cdots + m_{s} = n_{1} + \cdots + n_{t}$.
\end{definition}

The term `polygonal equality' arises from the work of Elsner et al.~\cite{EHKMZ} who, in studying a conjecture due to Bernius and Blanchard, characterized the finite subsets
$\{\vecx_i\}_{i = 1}^n$ and $\{\vecy_j\}_{j = 1}^n$
of $L_1(\Omega,\mu)$ for which one obtains
\begin{eqnarray}\label{Elsner}
  \sum_{i,j} \norm{\vecx_i - \vecy_j}_1 & = &
  \sum_{i< j} \norm{\vecx_i - \vecx_j}_1 + \sum_{i< j} \norm{\vecy_i - \vecy_j}_1.
  \end{eqnarray}
        
The existence of suitable $p$-polygonal equalities in $(X,d)$ has close connections with the equivalent concepts of $p$-negative type and generalized roundness (see for example \cite{LTW}) and so provides information concerning classical embedding questions.
For instance, Weston \cite{W2017} applied $p$-polygonal
equalities in $L_p$ spaces to show that any two-valued Schauder
basis of $L_p$ must have strict $p$-negative type.
As noted in \cite[Corollary 3.1.3]{W2017}, this property
obstructs the existence of certain types of isometry into
$L_p$ spaces.

In any metric space one can easily obtain `trivial' $p$-polygonal equalities. For example one could make all the weights zero, or choose $x_1 = y_1$, $x_2 = y_2$ and $m_1= m_2 = n_1 = n_2 = 1$. To rule out these uninteresting cases, Kelleher et al.~\cite{KMOW}  introduced the concept of a nontrivial $p$-polygonal equality and proved the following theorem.

\begin{theorem}\label{str-iff-poly}\cite[Lemma 3.21]{KMOW}
Suppose that $(X,d)$ is a metric space of $p$-negative type. Then $(X,d)$ is of strict $p$-negative type if and only if it admits no nontrivial $p$-polygonal equality. 
\end{theorem}

The property of being of strict $p$-negative type (and particularly the property of being of strict $1$-negative type), which arose in the study of isometric embeddings of metric spaces,
has now been linked to properties in many different areas of mathematics, from hyperbolic geometry to graph theory \cite{DWe,DWe2,HLMT,HKM,LaP,LiW}.

Theorem~\ref{str-iff-poly} left open the question of identifying the set 
  \[ P_{(X, d)} = \{p \ge 0 \st \text{$(X,d)$ admits a nontrivial $p$-polygonal equality}\}.  \]

It is the aim of this short note to remove the negative type assumption in Theorem~\ref{str-iff-poly} and to thereby identify the set $P_X = P_{(X, d)}$. Our main result is the following.

\begin{theorem}\label{main}
    Suppose that $(X,d)$ is a metric space and that $p \ge 0$. Then $(X,d)$ is of
    strict $p$-negative type if and only if $(X,d)$ does not admit any nontrivial $p$-polygonal equalities.
\end{theorem}

It is a classical result of Schoenberg
\cite{Sc2} that the set of all $p$ for which a metric space
$(X,d)$ is of $p$-negative type is always a closed interval
of the form $[0, \wp]$, where $\wp \geq 0$, or $[0, \infty)$.
Included here is the possibility that $\wp = 0$, in which
case $[0, \wp] = \{ 0 \}$. However, for finite
metric spaces, it is necessarily the case that $\wp > 0$.
This result is due to Deza and Maehara \cite{Dez} and,
independently, Weston \cite{W1995}. As an application of
Theorem~\ref{main} we show that the set $P_X$ is always a
closed interval of the form
$[\wp, \infty)$, where $\wp > 0$, or the empty set.

In the next section we shall give
the main definitions needed. In Section \ref{sect-3} we shall
give the proof of Theorem~\ref{main}
and examine some of the consequences of this result.

\section{Background and definitions}\label{sect-2}

Unless stated otherwise we shall assume that all metric spaces considered below are finite and have at least two elements. 
Given such a space $X = \{x_1,\dots,x_m\}$ with metric $d$ we can form its \textbf{distance matrix}
$D = \bigl(d(x_i,x_j)\bigr)_{i,j = 1}^{m}$. 
For $p \ge 0$, the \textbf{$p$-distance matrix} for $(X,d)$ is
$D_p = \bigl(d(x_i,x_j)^p\bigr)_{i,j = 1}^{m}$.
Let $F_0$ denote the hyperplane of all vectors
$\bxi = (\xi_1, \dots, \xi_m)^{T} \in \mR^m$ such that $\sum_{i=1}^m \xi_i = 0$. The usual inner product in $\mR^m$ will be denoted by $\ip<\cdot,\cdot>$.

\begin{definition}\label{neg-type}
Let $(X,d)$ be a finite metric space with $m$ points $\{x_{1}, \ldots, x_{m} \}$
and suppose that $p \ge 0$.
We say that $(X,d)$ is of \textbf{$p$-negative type} if
  \begin{equation}\label{p-neg}
    \ip<D_p \bxi,\bxi> = \sum_{i,j} d(x_i,x_j)^p \xi_i \xi_j \le 0
  \end{equation}
for all $\bxi = (\xi_1,\dots,\xi_{m})^{T} \in F_0$.
If strict inequality holds in (\ref{p-neg}) except when $\bxi = \zero$ we say that $(X,d)$ is of \textbf{strict $p$-negative type}.

The \textbf{supremal $p$-negative type of} $(X, d)$
is defined to be the quantity
\[
\wp(X) = \sup \{p \geq 0 \st \text{$(X,d)$ is of $p$-negative
type}\}.
\]
\end{definition}

By extension, a general metric space $(X,d)$ is said to be of
$p$-negative type if every finite metric subspace of $X$ is of
$p$-negative type. Similarly it is of
strict $p$-negative type if every finite metric subspace of $X$
has that property.

It is a classical result of Schoenberg \cite{Sc3} that $(X,d)$ is of $p$-negative type if and only if $(X,d^{p/2})$ embeds isometrically in a Hilbert space. Schoenberg showed that $(X,d)$ is of $p$-negative type for all $p \le \wp(X)$, and so $(X,d)$ embeds isometrically in a Hilbert space if and only if $\wp(X) \ge 2$.

Li and Weston \cite{LiW} showed that for
any finite metric space $(X,d)$ we have:
\begin{enumerate}[(i)]
  \item $(X,d)$ is of strict $p$-negative type for all $p \in [0,\wp(X))$;
  \item If $\wp(X) < \infty$, then $(X,d)$ is of
  $\wp(X)$-negative type but not of strict
  $\wp(X)$-negative type.
\end{enumerate}

The hypothesis that $(X,d)$ is a finite metric space is necessary in (ii); there are infinite metric spaces for which are of strict $\wp(X)$-negative type (see \cite[Example 2]{DWe} and
\cite[Section~5]{LiW}).

The case where $\wp(X) = \infty$ for a general metric space
$(X, d)$ is well understood. Faver et al.\ \cite{Fav}
have shown that $\wp(X) = \infty$ if and only if
$(X, d)$ is an ultrametric space.

In what follows, $(X, d)$ is assumed to be a finite metric
space, unless stated otherwise.

The main step in proving Theorem~\ref{main} is to replace the original form for a $p$-polygonal equality (\ref{p-poly}) with an expression involving the $p$-distance matrix for $(X,d)$. 

A \textbf{signed $(s,t)$-simplex} in $(X,d)$ is a collection of points $x_1,\dots,x_s,y_1,\dots,y_t \in X$ and associated real weights $m_1,\dots,m_s,n_1,\dots,n_t$. We shall denote such a simplex by $Q = [x_1(m_1),\dots,x_s(m_s);y_1(n_1),\dots,y_t(n_t)]
= [x_i(m_i);y_j(n_j)]_{s,t}$. 
It was shown in \cite[Section 3]{KMOW} that condition~(\ref{p-neg}) can be reformulated in terms of the $p$-simplex gaps for such simplices.

The \textbf{$p$-simplex gap} for a signed simplex $Q = [x_i(m_i);y_j(n_j)]_{s,t}$ is the quantity
  \[ \gamma_p(Q)
  = \sum_{i,j} m_i n_j d(x_i,y_j)^p
    -\sum_{1 \le i_1<i_2 \le s} m_{i_1} m_{i_2} d(x_{i_1},x_{i_2})^p
     - \sum_{1 \le j_1<j_2 \le t} n_{j_1} n_{j_2} d(y_{j_1},y_{j_2})^p.\]

In this language, a $p$-polygonal equality is an equality of the form $\gamma_p(Q) = 0$ for some signed
$(s,t)$-simplex $Q$ in $X$.  

\begin{definition}
A signed simplex $Q = [x_i(m_i);y_j(n_j)]_{s,t}$ is
said to be \textbf{completely refined} if the points $x_1,\dots,x_s,y_1,\dots,y_t$ are distinct and the real weights $m_1,\dots,m_s$, $n_1,\dots,n_t$ are all strictly positive.
\end{definition}

Suppose that $Q = [x_i(m_i);y_j(n_j)]_{s,t}$ is a completely refined simplex in $(X,d)$. We can list the elements of $X$ as $x_1,\dots,x_s,y_1,\dots,y_t,z_1,\dots,z_u$ and form
the vector
 \[ \bxi = (m_1,\dots,m_s,-n_1,\dots,-n_t,0,\dots,0)^{T},\] 
which is necessarily a nonzero vector in $F_0$. An elementary calculation then shows that for all $p \ge 0$,
  \begin{equation}\label{link}
 \ip<D_p \bxi,\bxi> = -2 \gamma_p(Q).
  \end{equation}
Conversely, starting from any nonzero $\bxi \in F_0$ one can construct a completely refined simplex $Q$ satisfying (\ref{link}). 

There is a useful dichotomy for signed simplices.
A key result in \cite{KMOW} shows that each signed simplex
$Q$ in $X$ may be algorithmically reduced to a so-called
\textbf{degenerate} signed simplex $Q^{\ast}$ (one in which
all weights are zero) or a completely refined signed
simplex $Q^{\ast}$. In either case, one has
$\gamma_{p}(Q) = \gamma_{p}(Q^{\ast})$ for all $p \geq 0$.
Signed simplices
that can be algorithmically reduced to a completely refined
signed simplex are said to be \textbf{nondegenerate}.
By definition, a \textbf{nontrivial $p$-polygonal equality}
in $(X, d)$ is an equality of the form $\gamma_{p}(X) = 0$
for some nondegenerate signed simplex $Q$ in $X$.
In what follows we shall make use of the following
characterization of nontrivial $p$-polygonal equalities
that is implicit in \cite[Section 3]{KMOW}.

\begin{theorem}\label{equiv-form}
Let $(X,d)$ be a finite metric space and suppose that $p \ge 0$. Then the following conditions are equivalent.
\begin{enumerate}[(i)]
  \item $(X,d)$ admits a nontrivial $p$-polygonal equality.
  \item $\ip<D_p \bxi,\bxi> = 0$ for some nonzero $\bxi \in F_0$.
\end{enumerate}
\end{theorem}

\begin{remark}
    In general, finding a nonzero $\bxi \in F_0$ such that $\ip<D_p \bxi,\bxi> = 0$ can be a complicated task. However,
    the situation that arises
    when $p = \wp(X) < \infty$ is considerably
    simpler. S{\'a}nchez \cite[Corollary 2.5]{San}
    showed that if  $D_{\wp(X)}$  is singular then $\ker(D_{\wp(X)}) \subseteq F_0$ and so any nonzero
    $\bxi \in \ker(D_{\wp(X)})$ satisfies $\ip<D_{\wp(X)} \bxi,\bxi> = 0$.
    On the other hand, if $D_{\wp(X)}$  is invertible,  we can take $\bxi = D_{\wp(X)}^{-1} (1,1,\dots,1)^T$. Then $\zero \ne \bxi \in F_0$ and $\ip<D_{\wp(X)} \bxi,\bxi> = 0$.
\end{remark}

\section{Proofs}\label{sect-3}

In light of Theorem~\ref{equiv-form} we can rephrase Theorem~\ref{main} in the case of finite spaces as follows.

\begin{theorem}\label{main2}
Suppose that  $(X,d) = (\{x_1, x_2,\dots, x_m\},d)$ is a finite metric space with $m \ge 2$ and that $p \ge 0$. Then the following are equivalent.
\begin{enumerate}[(i)]
  \item $(X,d)$ is of strict $p$-negative type.
  \item $(X,d)$ admits no nontrivial $p$-polygonal equalities.
  \item $\ip<D_p \bxi,\bxi> \ne 0$ for all nonzero $\bxi \in F_0$.
\end{enumerate}
\end{theorem}

\begin{proof} By Theorem~\ref{equiv-form} it suffices to prove that (i) and (iii) are equivalent. 

It follows from the definitions that if $(X,d)$ is of strict negative type then (iii) holds, while if $(X,d)$ is of $p$-negative type, but not strict $p$-negative type, then there exists a nonzero $\bxi \in F_0$ such that $\ip<D_p \bxi,\bxi> = 0$. The issue then is whether it is possible to have a space which is not of $p$-negative type for some $p$, and which does not admit any nontrivial $p$-polygonal equalities.

Suppose that $(X,d)$ is not of $p$-negative type. Then there is some (necessarily) nonzero $\bxi_1 = (\alpha_1,\alpha_2,\dots,\alpha_m)^{T} \in F_0$ such that $\ip<D_p \bxi_1,\bxi_1> > 0$. On the other hand, if we let $\bxi_0 = (1,-1,0,\dots,0)^{T}$, then $\ip<D_p \bxi_0,\bxi_0> = -2d(x_1, x_2)^p < 0$. For $0 < t < 1$ let $\bxi_t = (1-t)\bxi_0 +t \bxi_1$. Then certainly $\bxi_t \in F_0$ for all $t$, and by the Intermediate Value Theorem, there exists $s \in (0,1)$ such that $\ip<D_p \bxi_s,\bxi_s> = 0$. It just remains to check that $\bxi_s \ne \zero$. 

Suppose then that 
  \[ \bxi_s = (1-t+t \alpha_1,t-1+t \alpha_2,t\alpha_3,\dots,t\alpha_m)^{T}  = \zero.\] 
This immediately implies that $\alpha_j = 0$ for $3 \le j \le m$. As $\bxi_1 \in F_0$ this means that $\alpha_2 = -\alpha_1$. But if this is the case then $\ip<D_p \bxi_1,\bxi_1> < 0$ which is a contradiction. Thus $\bxi_s \ne \zero$ and so we have a nontrivial $p$-polygonal equality in $(X,d)$. 
\end{proof}

Since the existence of a nontrivial $p$-polygonal equality is also determined by what happens on the finite subsets of $X$, Theorem~\ref{main} follows immediately from Theorem~\ref{main2}.

\begin{corollary}\label{poly-int}
If $(X,d)$ is a finite metric space with $\wp(X)$ finite, then 
  \[ P_X = \{p \ge 0 \st \text{$(X,d)$ admits a nontrivial $p$-polygonal equality}\} = [\wp(X),\infty) .\]
\end{corollary}

\begin{proof}
By Li and Weston \cite{LiW}, $(X,d)$ is of strict $p$-negative
type for all $p < \wp(X)$ and $(X,d)$ admits a nontrivial
$\wp(X)$-polygonal equality. Now apply Theorem~\ref{main2}
for all $p \not= \wp(X)$.
\end{proof}

\begin{remark}
In the case of an infinite metric space
$(X,d)$ with $\wp(X) < \infty$ we cannot, in general, draw
the exactly same conclusion as that given in Corollary
\ref{poly-int}. This is because there exist infinite metric
spaces that are of strict supremal $p$-negative type. An
example of this phenomenon is the infinite metric tree given in Doust and Weston \cite[Example 2]{DWe}. So in this case
the best that we can say is that $P_X$ is one of the intervals
$[\wp(X), \infty)$ or $(\wp(X), \infty)$. An exception to
this arises when $\wp(X) = 0$, wherein one necessarily obtains
$P_{X} = (0, \infty)$. The first example of a metric space
$(X, d)$ with $\wp(X) = 0$ was constructed by Enflo \cite{En1}.
A more recent result of Dahma and Lennard \cite{Dah} showed
that the Schatten $q$-class $\mathcal{C}_{q}$, $q \not= 2$,
does not have $p$-negative type for any $p > 0$. In other
words, $\wp(\mathcal{C}_{q}) = 0$ for all $q \not= 2$. It
follows from Corollary \ref{poly-int} that $\mathcal{C}_{q}$,
$q \not= 2$, admits a nontrivial $p$-polygonal equality
for all $p > 0$. This settles an open question raised by
Kelleher et al.\ \cite[Section 6]{KMOW}.
\end{remark}

\begin{remark}
A simple example of a finite metric space $(X, d)$ for which
$P_X$ is the empty set is a two point discrete metric space.
For all $p \ge 0$, $\ip<D_p \bxi,\bxi> = -2 \xi_1^2$ for $\bxi \in F_0$,
and so one may apply Theorem~\ref{main2}. More generally,
Faver et al.\ \cite[Section 5]{Fav} have shown that $P_{(X, d)}$
is the empty set if and only if $(X, d)$ is an ultrametric space.
\end{remark}

\bibliographystyle{amsalpha}

\end{document}